\documentclass[11pt]{article}

\usepackage{amssymb}
\usepackage{amsthm}
\usepackage{amsmath,amsfonts,amssymb}
\usepackage{graphicx}
\usepackage{algpseudocode}
\usepackage{algorithm}
\usepackage{natbib}

\newtheorem{theorem}{Theorem}[section]
\newtheorem{lemma}[theorem]{Lemma}
\newtheorem{proposition}[theorem]{Proposition}

\newtheorem{definition}[theorem]{Definition}
\newtheorem{remark}[theorem]{Remark}

\begin{document}

\title{An Infinite-Dimensional Variational Inequality Formulation and Existence Result for Dynamic User Equilibrium with Elastic Demands}
\author{Ke Han$^{a}\thanks{Corresponding author, e-mail: kxh323@psu.edu}$
\qquad Terry L. Friesz$^{b}\thanks{e-mail: tfriesz@psu.edu}$
\qquad Tao Yao$^{b}\thanks{e-mail: tyy1@engr.psu.edu}$ \\\\
$^{a}$\textit{Department of Mathematics, }\\
\textit{Pennsylvania State University, University Park, PA 16802, USA}\\
$^{b}$\textit{Department of  Industrial and Manufacturing Engineering, }\\
\textit{Pennsylvania State University, University Park, PA 16802, USA}}
\date{}
\maketitle

\begin{abstract}
\noindent This paper is concerned with {\it dynamic user equilibrium} (DUE) with elastic travel demand (E-DUE). We present and prove a {\it variational inequality} (VI) formulation of E-DUE using measure-theoretic argument. Moreover,  existence  of the E-DUE is formally established  with a version of Brouwer's fixed point theorem in a properly defined Hilbert space. The existence proof requires the effective delay operator to be continuous, a regularity condition also needed to ensure the existence of DUE with fixed demand \citep{existence}. Our proof  does not invoke the {\it a priori} upper bound of the departure rates (path flows). 

\medskip

\noindent \textbf{Keywords:} \textit{dynamic user equilibrium; elastic demand; variational inequalities; differential variational inequalities; existence}
\end{abstract}

\section{\label{Intro}Introduction}

This paper is concerned with an elastic demand extension of the fixed-demand dynamic
traffic assignment model originally presented in \cite{Friesz1993} and
discussed subsequently by \cite{FM2006, FKKR, FHNMY, FM2013}. As such, it is
concerned with a specific type of dynamic traffic assignment known as
dynamic user equilibrium (DUE) for which travel cost, including delay as
well as early and late arrival penalties, are equilibrated and demand is
determined endogenously.

\subsection{Introductory remarks on dynamic user equilibrium}

In the past two decades there have been many efforts to develop a theoretically sound formulation of dynamic user equilibrium especially in continuous time. As is pointed out by \cite{FKKR}, DUE models tend to be comprised of five essential components:
\begin{enumerate}
\item a model of path delay;
\item flow dynamics;
\item flow propagation constraints;
\item a route and departure time choice model; and
\item a model of demand growth (evolution).
\end{enumerate}

Dynamic user equilibrium models from the early 1990s forward have been largely concerned with the sub-models 1 through 4 above, concentrating on the within-day time scale for which travelers make routing and departure time decisions. Items 1 through 3 above form a sub-model known as the {\it dynamic network loading} (DNL), which determines  arc- and/or path-specific volumes and flows as well as experienced path delays when departure rates (path flows) are known for each path. Item 4 aims at expressing mathematically the notion of Nash-like equilibrium conditions. Item  5, demand evolution, occurs on the day-to-day time scale and allows travel demands to be updated.

There are multiple means of expressing the Nash-like notion of a dynamic user equilibrium, including the following:
\begin{itemize}
\item[(i)] an infinite-dimensional variational inequality \citep{Friesz1993, SW1, SW2}
\item[(ii)] a nonlinear complementarity problem \citep{WTC, HUD}
\item[(iii)] a differential variational inequality \citep{Friesz2001, FHNMY, FKKR, FM2006}
\item[(iv)] a differential complementarity system \citep{Pang}
\end{itemize}

The variational inequality representation  is presently the primary  mathematical form employed for dynamic user equilibrium. One of the advantages of such variational inequality formulation is that it subsumes almost all simultaneous route-and-departure time choice DUE models regardless of the dynamic network loading procedure employed. In fact, the arc dynamics and flow propagation constraints can be naturally embedded in the {\it effective delay operator} which is viewed as a mapping between two Hilbert spaces. However, it would be a mistake to think that somehow the VI formulation is an ``easier" formulation since the effective delay operators are generally not knowable in closed form; in fact, the delay operators may be non-analytic and may need to be derived from an embedded delay model, data combined with response surface methodology, or data combined with inverse modeling. Analytical results of the effective delay operator is crucial for understanding qualitative properties of dynamic user equilibrium such as existence, uniqueness and convergence of certain computational schemes. The existence of DUE requires that the delay operator is continuous, which is a consequence of a generalization of Brouwer's fixed point theorem \citep{Browder}. In addition, the uniqueness of a DUE solution further requires that the delay operator is strongly monotone \citep{Nagurney}. \cite{FM2006} propose and test a fixed point algorithm implemented in continuous time to solve the differential variational inequality (DVI) formulation of DUE \citep{Friesz2001}; that algorithm requires monotonic effective delay operators to assure convergence.

\subsection{Some history of DUE with elastic demand}

In this section we review some of the few prior efforts to model DUE with
elastic travel demand. This review is based in part on \cite{FM2013}. Most of the studies of DUE reported in the DTA literature are about
dynamic user equilibrium with constant travel demand for each
origin-destination pair. We denote such a dynamic user equilibrium with fixed demand by F-DUE. It is, of course, not generally true that travel
demand is fixed, even for short time horizons. \cite{ADL} and \cite{YH} directly consider elastic travel demand; their work possesses a
limited relationship to the analysis presented in this paper, for their work
is concerned with a simple bottleneck instead of a nontrivial network, which
is our focus.

\cite{YM} extend a simple bottleneck model to a general queuing
network with known elastic demand functions for each origin-destination (OD)
pair. \cite{WTC} study a version of dynamic user equilibrium with
elastic demand, using a complementarity formulation that requires path
delays to be expressible in closed form. \cite{SL} study dynamic
user equilibrium with elastic travel demand when network loading is based on
the {\it cell transmission model} (CTM); their formulation is based on discrete
time and is expressed as a finite-dimensional variational inequality (VI).
\cite{HUD} study dynamic user equilibrium with elastic travel demand
for a network whose traffic flows are also described by CTM.

Although \cite{FKKR} show that analysis and computation of dynamic
user equilibrium with constant travel demand is tremendously simplified by
stating it as a a differential variational inequality (DVI), they do not
discuss how elastic demand may be accommodated within a DVI framework. \cite{FM2013} extend the DVI formulation to an elastic demand setting. Such a DVI formulation of elastic demand DUE is not straightforward. In particular, the DVI presented therein has both infinite-dimensional and finite-dimensional terms; moreover, for any given origin-destination pair,
inverse travel demand corresponding to a dynamic user equilibrium depends on
the terminal value of a state variable representing cumulative departures. The DVI formulation achieved in that paper is significant because it allows
the still emerging theory of differential variational inequalities to be
employed for the analysis and computation of solutions of the elastic-demand
DUE problem when simultaneous departure time and route choice are within the
purview of users, all of which constitutes a foundation problem within the
field of dynamic traffic assignment.

A good review of recent insights into
abstract differential variational inequality theory, including computational
methods for solving such problems, is provided by \cite{PS}.
Also, differential variational inequalities involving the kind of explicit,
agent-specific control variables employed herein are presented in \cite{DODG}.

\subsection{Main results}

This paper further advances the knowledge of continuous-time E-DUE based on \cite{FM2013} in terms of formulation and qualitative properties. In particular, we state and prove an equivalent variational inequality (VI) formulation of the E-DUE problem. The proof uses a measure-theoretic argument, and does not invoke the optimal control framework which is the primary methodology employed by \cite{FM2013} to establish the DVI formulation.  Such a result is presented  in Theorem \ref{vielasticthm}.

As an immediate application of the proposed variational inequality formulation, we will  analyze and establish existence result for the E-DUE problem. As commented in \cite{existence}, the most obvious approach to establishing existence is to convert the problem to an equivalent variational inequality problem and then apply a version of Brouwer's fixed point existence theorem. Nearly all proofs of DUE existence employ Brouwer's fixed point theorem, either implicitly or explicitly.  One statement of Brouwer's theorem appears as Theorem 2 of \cite{Browder}. Approaches based on Brouwer's theorem require the set of feasible path flows under consideration to be compact and convex in a topological vector space, and typically involve the {\it a priori} bound on path flows. For instance, using the {\it link delay model} introduced by \cite{Friesz1993}, \cite{ZM} show that a {\it route choice} (RC) dynamic user equilibrium exists under certain regularity conditions. In their modeling framework, the departure rate at each origin is given as {\it a priori} and assumed to be bounded from above. Thus one is assured that all path flows are automatically uniformly bounded. In \cite{WTC}, the existence of an arc-based user equilibrium is established under the assumption that the path flows are {\it a priori} bounded.

Difficulties arise in the proof of a general existence theorem from two aspects: (i) in a continuous-time setting, the set of feasible path flows is not compact; and (ii) the assumption  of {\it a priori} boundedness of path flows, which is usually required for a topological argument, does not arise from any behavioral argument or theory. The existence proof provided by this paper manages to overcome these two major difficulties.  Regarding item (i) above, we employ successive finite-dimensional approximations of the feasible path flows set, which allows Brouwer's fixed point theorem to be applied. Regarding item (ii), we propose an in-depth analysis and computation of the path flows under minor assumptions on the travelers' disutility functions.

Existence result for the elastic demand case is further complicated by the fact that the corresponding VI formulation has both infinite-dimensional and finite-dimensional terms (see Theorem \ref{vielasticthm} below). In order to apply Browder's theorem \citep{Browder}, one needs to work in an extended Hilbert space that is a product of an infinite-dimensional space and a finite-dimensional space, and define appropriate inner product that allows compactness and weak topology to be properly defined. It is significant that our existence result for E-DUE, stated and proved in Theorem \ref{existencethm}, does not rely on the {\it a priori} upper bound of path flows and can be established for any dynamic network loading sub-model with reasonable and weak regularity conditions.

\subsection{Organization} The rest of this paper is organized as follows. Section \ref{secbackground} recaps several key notations and concepts that are essential for the subsequent discussion. In Section \ref{secvielastic}, we  present one of the main results of this paper: the variational inequality (VI) formulation of the simultaneous route-and-departure choice dynamic user equilibrium with elastic travel demand (E-DUE). Section \ref{secdvi} briefly reviews the differential variational inequality formulation proposed by \cite{FM2013}. Section \ref{secexistenceedue} establishes  the existence result for E-DUE based on the VI formulation proposed in Section \ref{secvielastic}. The proof overcomes several difficulties known to other researchers, including non-compactness of the feasible set and the {\it a priori} boundedness of path flows.

\section{Notation and essential background}\label{secbackground}

The time interval of analysis is a single commuting period or
\textquotedblleft day\textquotedblright\ expressed as 
\begin{equation*}
\left[ t_{0},t_{f}\right] \subset \Re ^{1}
\end{equation*}%
where $t_{f}>t_{0}$, and both $t_{0}$ and $t_{f}$ are fixed. Here, as in all
DUE modeling, the single most crucial ingredient is the path delay operator,
which provides the delay on any path $p$ per unit of flow departing from the
origin of that path; it is denoted by%
\begin{equation*}
D_{p}(t,h)\text{ \ \ \ }\forall p\in \mathcal{P}
\end{equation*}%
where $\mathcal{P}$ is the set of all paths employed by travelers, $t$
denotes departure time, and $h$ is a vector of departure rates. From these,
we construct effective unit path delay operators $\Psi _{p}(t,h)$ by adding
the so-called schedule delay $f\left( t+D_{p}(t,h)-T_{A}\right) $; that is%
\begin{equation}\label{cost}
\Psi _{p}(t,h)=D_{p}(t,h)+f\left( t+D_{p}(t,h)-T_{A}\right) \qquad \forall p\in \mathcal{P}
\end{equation}%
where $T_{A}$ is the desired arrival time and $T_{A}<t_{f}$. The function $%
f\left( \cdot \right) $ assesses a penalty whenever%
\begin{equation}
t+D_{p}(t,h)\neq T_{A}  \label{pgt}
\end{equation}%
since $t+D_{p}(t,h)$ is the clock time at which departing traffic arrives at
the destination of path $p\in \mathcal{P}$. We stipulate that each
\begin{equation*}
\Psi _{p}(\cdot ,h):\left[ t_{0},t_{f}\right]~ \longrightarrow ~\Re _{++}^{1}%
\qquad \forall p\in \mathcal{P}
\end{equation*}%
is measurable and strictly positive. We employ the obvious notation
\begin{equation*}
\big( \Psi _{p}(\cdot ,h):p\in \mathcal{P}\big) \in \Re ^{|\mathcal{P}|}
\end{equation*}
to express the complete vector of effective delay operators.

It is now well known that path delay operators may be obtained from an
embedded delay model, data combined with response surface methodology, or
data combined with inverse modeling. Unfortunately, regardless of how
derived, realistic path delay operators do not possess the desirable
property of monotonicity; they may also be non-differentiable.

\subsection{Review on dynamic user equilibrium with fixed demand (F-DUE)}
For the completeness of our presentation, in this section we recap  the definition of DUE with fixed travel demand, originally articulated by \cite{Friesz1993}. 

Let us introduce the fixed trip matrix $\big(Q_{ij}: (i,\,j)\in\mathcal{W}\big)$, where each $Q_{ij}\in \Re _{++}$ is the fixed travel demand between origin-destination pair $\left( i,j\right) \in \mathcal{%
W}$, where $\mathcal{W}$ is the set of origin-destination pairs. Note that $Q_{ij}$ represents traffic volume, not flow. Finally we let $\mathcal{P}_{ij}\subset \mathcal{P}$ be the set of
paths connecting origin-destination pair $\left( i,j\right) \in \mathcal{W}$.  As mentioned earlier, $h$ is the vector of path flows $h=\{h_p: p\in\mathcal{P}\}$. 
We denote the space of square integrable functions on the real interval $[t_0,\,t_f]$ by $L^2[t_0,\,t_f]$. We stipulate that each path flow is square integrable, that is
$$
h\in\big(L_+^2[t_0,\,t_f]\big)^{|\mathcal{P}|}
$$
where $\big(L_+^2[t_0,\,t_f]\big)^{|\mathcal{P}|}$ is the positive cone of  the $|\mathcal{P}|$-fold product of the Hilbert space $L^2[t_0,\,t_f]$. Each element $h=(h_{p}:p\in \mathcal{P})\in \big(L_+^2[t_0,\,t_f]\big)^{|\mathcal{P}|}$ is interpreted as a vector of departure-time densities, or more simply path flows, measured at the entrance of the first arc of the relevant path.
It will be seen that these departure time densities are defined only up to a
set of measure zero. With this in mind, let $\nu $ denote a Lebesgue measure
on $[t_{0},\,t_{f}]$, and for each measurable set $S\subseteq %
[t_{0},\,t_{f}] $, let $\forall _{\nu }(t\in S)$ denote the phrase
\emph{for $\nu $-almost all $t\in S$}. If $S=[t_{0},\,t_{f}]$,
then we may at times simply write $\forall _{\nu }(t)$.

We write the flow conservation constraints as
\begin{equation}\label{cons}
\sum_{p\in P_{ij}}\int_{t_0}^{t_f}h_p(t)\,dt~=~Q_{ij}\qquad\forall (i,\,j)\in\mathcal{W}
\end{equation}
where \eqref{cons} consists of Lebesgue integrals. \footnote{The demand satisfaction expressed via Lebesgue integrals is not enough to assure that the path flows are bounded. This has been the major hurdle to proving existence without the {\it a priori} invocation of bounds on path flows.} Using the notation
and concepts we have mentioned, the feasible region for DUE when effective
delay operators are known is
\begin{equation}\label{chapVI:lambda}
\Lambda_0~=~\left\{ h\geq 0:\sum_{p\in \mathcal{P}_{ij}}%
\int_{t_{0}}^{t_{f}}h_{p}\left( t\right) dt=Q_{ij}\text{ \ \ \ }\forall
\left( i,j\right) \in \mathcal{W}\right\} \subseteq \left( L_{+}^{2}\left[
t_{0},t_{f}\right] \right) ^{\left\vert \mathcal{P}\right\vert }
\end{equation}

\noindent In order to define an appropriate concept of minimum travel costs in the present context, we require the measure-theoretic analog of the infimum of a set of numbers. In particular,  for any measurable set $S\subseteq [t_0,\,t_f]$ with $\nu(S)>0$, and any measurable function $f: S\rightarrow \Re$, the {\it essential infimum} of $f(\cdot)$ on $S$ is given by
\begin{equation}\label{chapVI:essinf}
\hbox{essinf}\left\{f(s):~s\in S\right\}~=~\sup\left\{ x\in\Re:~\nu\{s\in S:~f(s)<x\}~=~0\right\}
\end{equation}
Note that for each $x>\hbox{essinf}\{f(s):~s\in S\}$ it must be true by definition that 
$$
\nu\{s\in S:~f(s)<x\}>0
$$

\noindent Let us  define the essential infimum of effective travel delays
$$
v_p~=~v_p(h)~=~\hbox{essinf}\left\{\Psi_p(t,\,h):~t\in[t_0,\,t_f]\right\}~\geq~0\qquad\forall p\in\mathcal{P}
$$
$$
v_{ij}~=~v_{ij}(h)~=~\min\left\{v_p(h): ~p\in \mathcal{P}_{ij}\right\} \qquad \forall \left( i,j\right) \in \mathcal{W}
$$

\noindent The following definition of dynamic user equilibrium was first articulated by \cite{Friesz1993}:
\begin{definition}\label{duedef}
{\bf (Dynamic user equilibrium)} A vector of departure rates (path
flows) $h^{\ast }\in \Lambda_0$ is a dynamic user equilibrium if
\begin{align}\label{chapVI:comp}
h_{p}^{\ast }\left( t\right) >0,~p\in \mathcal{P}_{ij}~~\Longrightarrow~~ \Psi _{p}(t,\,h^{\ast }) ~=~v_{ij}(h^*)\qquad & \forall_{\nu}(t)\in[t_0,\,t_f], \qquad\forall (i,\,j)\in\mathcal{W}
\\
 \label{chapVI:nobetter}
\Psi_p(t,\,h^*)~\geq~v_{ij}(h^*) \qquad & \forall_{\nu}(t)\in[t_0,\,t_f], \qquad \forall (i,\,j)\in\mathcal{W}
\end{align}
We denote this equilibrium by $DUE\left( \Psi,\,\Lambda_0,\, [
t_{0},t_{f}] \right)$.
\end{definition}

Using measure theoretic arguments, \cite{Friesz1993} establish that
a dynamic user equilibrium is equivalent to the following variational
inequality under suitable regularity conditions:
\begin{equation} \label{duevi}
\left. 
\begin{array}{c}
\text{find }h^{\ast }\in \Lambda _{0}\text{ such that} \\ 
\sum\limits_{p\in \mathcal{P}}\displaystyle \int\nolimits_{t_{0}}^{t_{f}}\Psi _{p}(t,h^{\ast
})(h_{p}-h_{p}^{\ast })dt\geq 0 \\ 
\forall h\in \Lambda _{0}%
\end{array}%
\right\} VI(\Psi ,\Lambda_0 ,\left[ t_{0},t_{f}\right] ) 
\end{equation}%

\subsection{Dynamic user equilibrium with elastic demand (E-DUE)}
The general setup of DUE with elastic demand is similar to that of the fixed demand case, with the exception that the total travel demand $Q_{ij}$ between an origin-destination pair $(i,\,j)\in\mathcal{W}$ is no longer a prescribed constant.  Rather, transportation demand is assumed to be expressed as the following invertible function
$$
Q_{ij}~=~F_{ij}[v]
$$
for each origin-destination pair $(i,\,j)\in\mathcal{W}$, where $\mathcal{W}$ is the set of all origin-destination pairs and $v$ is a concatenation of origin-destination minimum travel cost $v_{ij}$ associated with $(i,\,j)\in\mathcal{W}$. That is, we have that
\begin{align*}
v_{ij}&~\in~\Re_{++}\\
v&~=~\big(v_{ij}:~(i,\,j)\in\mathcal{W}\big)\in\Re_+^{|\mathcal{W}|}
\end{align*}

\noindent Note that to say $v_{ij}$ is a minimum travel cost means it is the minimum
cost for all departure time choices and all route choices pertinent to
origin-destination pair $\left( i,j\right) \in \mathcal{W}$. Further note
that $Q_{ij} $ is the unknown cumulative travel demand
between $\left( i,j\right) \in \mathcal{W}$ that must ultimately arrive by
time $t_{f}$.

We will also find it convenient to form the complete vector of travel
demands by concatenating the origin-specific travel demands to obtain%
$$
Q ~=~\left( Q_{ij} :~( i,\, j) \in \mathcal{W}%
\right) \in \Re_+^{\left\vert \mathcal{W}\right\vert } 
$$
$$
F:~\Re _{++}^{\left\vert \mathcal{W}\right\vert }\longrightarrow \Re
_{+}^{\left\vert \mathcal{W}\right\vert },\qquad v~\mapsto~Q
$$
The inverse demand function for every $\left( i,j\right) \in \mathcal{W}$ is%
\begin{equation*}
v_{ij}=\Theta _{ij}\left[ Q \right] 
\end{equation*}%
and we naturally define%
\begin{equation*}
\Theta[Q] ~=~\big( \Theta _{ij}[Q]:\left( i,j\right) \in \mathcal{W}\big) \in \Re_{++}^{\left\vert \mathcal{W}\right\vert }
\end{equation*}

As a consequence, we employ the following  set of feasible departure flows when the travel demand between each origin-destination pair is unknown. 
\begin{equation}\label{tildelambdadef}
\widetilde\Lambda~=~\left\{(h,\,Q):~h\geq 0,~~ \sum_{p\in\mathcal{P}_{ij}}\int_{t_0}^{t_f} h_p(t)\,dt=Q_{ij}\quad \forall (i,\,j)\in\mathcal{W}\right\}
\subset\left(L^2[t_0,\,t_f]\right)^{|\mathcal{P}|}\times \Re_{+}^{|\mathcal{W}|}
\end{equation}
where $\left(L^2[t_0,\,t_f]\right)^{|\mathcal{P}|}\times \Re_{+}^{|\mathcal{W}|}$ is the direct product  of the $|\mathcal{P}|$-fold product of Hilbert spaces consisting of square-integrable path flows, and the $|\mathcal{W}|$-dimensional Euclidean space consisting of vectors of elastic travel demands. 

With preceding preparation, we are in a place where the simultaneous route-and-departure-choice dynamic user equilibrium with elastic demand can be rigorously defined, as follows.

\begin{definition}\label{dueelasticdef}
{\bf (Dynamic user equilibrium with elastic demand)} A pair $(h^*,\, Q^*)\in \widetilde \Lambda$, where $h^*$ is a vector of departure rates (path flows)  and $Q^*$ is the associated vector of travel demands, is said to be a dynamic user equilibrium with elastic demand if for all $(i,\,j)\in\mathcal{W}$,
\begin{align}
\label{chapVI:eqn1}
h_p^*(t)~>~0,~p\in\mathcal{P}_{ij}~\Longrightarrow~\Psi_p(t,\,h^*)~=~\Theta_{ij}[Q^*]\qquad  &\forall_{\nu}(t)\in[t_0,\, t_f]
\\
\label{chapVI:eqn2}
\Psi_p(t,\,h^*)~\geq~\Theta_{ij}[Q^*] \qquad &\forall_{\nu} (t)\in[t_0,\,t_f],\quad \forall p\in\mathcal{P}_{ij}
\end{align}

\end{definition}

\section{The variational inequality formulation of E-DUE}\label{secvielastic}

Experience with differential games suggests that the DUE problem with elastic demand can be expressed as a variational inequality, as shown in the theorem below. 

\begin{theorem}\label{vielasticthm} {\bf (E-DUE equivalent to a variational inequality)} Assume $\Psi_p(\cdot,\,h): [t_0,\,t_f]\rightarrow \Re_{++}$ is measurable and strictly positive for all $p\in\mathcal{P}$ and all $h$ such that $(h,\,Q)\in\widetilde \Lambda$. Also assume that the elastic travel demand function is invertible with inverse $\Theta_{ij}[\cdot]$ for all $(i,\,j)\in \mathcal{W}$. Then a pair, $(h^*,\,Q^*)\in\widetilde\Lambda$, is a DUE with elastic demand as in Definition \ref{dueelasticdef} if and only it solves the following variational inequality:
\begin{equation}\label{chapVI:elasticvi}
\left.\begin{array}{c}
\hbox{find} ~(h^{\ast},\,Q^*)\in \widetilde\Lambda~\hbox{such that}\\
\displaystyle\sum_{p\in \mathcal{P}}\int\nolimits_{t_{0}}^{t_{f}}\Psi _{p}(t,h^*)(h_{p}-h_{p}^{\ast })dt
-\sum_{(i,\,j) \in \mathcal{W}}\Theta _{ij}\left[ Q^{\ast } \right] \left[Q_{ij} -Q_{ij}^{\ast}
\right] ~\geq~0   \\ \forall (h,\,Q)\in \widetilde\Lambda  
\end{array} 
\right\} VI(\Psi,\Theta,t_{0},t_{f})
\end{equation}
\end{theorem}

\begin{proof}
(i)[\emph{Necessity}] Given a DUE solution with elastic demand $(h^*,\,Q^*)\in\widetilde\Lambda$, we easily deduce from \eqref{chapVI:eqn1} and \eqref{chapVI:eqn2} that for any $(h,\,Q)\in\widetilde\Lambda$, 
\begin{align}
&\sum_{p\in \mathcal{P}}\int\nolimits_{t_{0}}^{t_{f}}\Psi _{p}(t,h^*)(h_{p}-h_{p}^{\ast })dt
-\sum_{(i,\,j) \in \mathcal{W}}\Theta _{ij}\left[ Q^{\ast } \right] \left[Q_{ij} -Q_{ij}^{\ast}
\right]\nonumber\\
~=~& \sum_{(i,\,j)\in\mathcal{W}}\left(\sum_{p\in\mathcal{P}_{ij}}\int_{t_0}^{t_f}\Psi_p(t,\,h^*)h_p(t)\,dt-\Theta_{ij}[Q^*]\cdot Q_{ij}\right)\nonumber\\
&~-~\sum_{(i,\,j)\in\mathcal{W}}\left(\sum_{p\in\mathcal{P}_{ij}}\int_{t_0}^{t_f}\Psi_p(t,\,h^*)h^*_p(t)\,dt-\Theta_{ij}[Q^*]\cdot Q^*_{ij}\right)\nonumber\\
~=~&\sum_{(i,\,j)\in\mathcal{W}}\left(\sum_{p\in\mathcal{P}_{ij}}\int_{t_0}^{t_f}\Psi_p(t,\,h^*)h_p(t)\,dt-\Theta_{ij}[Q^*]\cdot Q_{ij}\right)\nonumber\\
&~-~\sum_{(i,\,j)\in\mathcal{W}}\left(\sum_{p\in\mathcal{P}_{ij}}\Theta_{ij}[Q^*]\cdot\int_{t_0}^{t_f}h^*_p(t)\,dt- \Theta_{ij}[Q^*]\cdot Q^*_{ij}\right)\nonumber\\
~=~& \sum_{(i,\,j)\in\mathcal{W}}\left(\sum_{p\in\mathcal{P}_{ij}}\int_{t_0}^{t_f}\Psi_p(t,\,h^*)h_p(t)\,dt-\Theta_{ij}[Q^*]\cdot Q_{ij}\right)\nonumber\\
&~-~\sum_{(i,\,j)\in\mathcal{W}}\Theta_{ij}[Q^*]\cdot\left(\sum_{p\in\mathcal{P}_{ij}}\int_{t_0}^{t_f}h^*_p(t)\,dt- Q^*_{ij}\right)\nonumber\\
~=~& \sum_{(i,\,j)\in\mathcal{W}}\left(\sum_{p\in\mathcal{P}_{ij}}\int_{t_0}^{t_f}\Psi_p(t,\,h^*)h_p(t)\,dt-\Theta_{ij}[Q^*]\cdot Q_{ij}\right)\label{chapVI:eqn3}
\end{align}
Observe that in \eqref{chapVI:eqn3}, given any $(i,\,j)\in\mathcal{W}$,
\begin{align}
&\sum_{p\in\mathcal{P}_{ij}}\int_{t_0}^{t_f}\Psi_p(t,\,h^*)h_p(t)\,dt-\Theta_{ij}[Q^*]\cdot Q_{ij}\nonumber\\
~\geq~& v_{ij}^*\sum_{p\in\mathcal{P}_{ij}}\int_{t_0}^{t_f}h_p(t)\,dt-\Theta_{ij}[Q^*]\cdot Q_{ij}~=~v_{ij}^*\left(\sum_{p\in\mathcal{P}_{ij}}\int_{t_0}^{t_f}h_p(t)\,dt-Q_{ij}  \right)~=~0 \label{chapVI:eqn4}
\end{align}
where $v_{ij}^*$ is the essential infimum of $\Psi_p(\cdot,\, h^*)$ on $[t_0,\,t_f]$ for all $p\in\mathcal{P}_{ij}$, and is equal to $\Theta_{ij}[Q^*]$ according to \eqref{chapVI:eqn1}. 

As an immediate consequence of  \eqref{chapVI:eqn3} and \eqref{chapVI:eqn4}, the following inequality holds for all $(h,\,Q)\in\widetilde\Lambda$.
\begin{equation}\label{chapVI:eqn5}
\sum_{p\in \mathcal{P}}\int\nolimits_{t_{0}}^{t_{f}}\Psi _{p}(t,h^*)(h_{p}-h_{p}^{\ast })dt
-\sum_{(i,\,j) \in \mathcal{W}}\Theta _{ij}\left[ Q^{\ast } \right] \left[Q_{ij} -Q_{ij}^{\ast}
\right]~\geq~0
\end{equation}

(ii)[\emph{Sufficiency}] Assume that \eqref{chapVI:eqn5} holds for any $(h,\,Q)\in\widetilde\Lambda$, then $h^*$ is a solution to the DUE problem with fixed demand given by $Q^*$, since in this case the second term in the left hand side of \eqref{chapVI:eqn5} vanishes and we recover the well-known VI for the fixed demand case, see \eqref{duevi}. Therefore by Definition
\ref{duedef}, for any $(i,\,j)\in\mathcal{W}$,
$$
\begin{cases}
h_{p}^{\ast }\left( t\right) >0,~p\in \mathcal{P}_{ij}~~\Longrightarrow~~ \Psi _{p}\left[
t,h^{\ast }\left( t\right) \right] ~=~v^*_{ij}\\
\Psi_p(t,\,h^*)~>~v^*_{ij}~~\Longrightarrow~~h_p^*(t)~=~0
\end{cases}\qquad\forall_{\nu}(t)\in[t_0,\,t_f]
$$
In order to show that $(h^*,\,Q^*)$ is a DUE with elastic demand using definition \eqref{chapVI:eqn1} and \eqref{chapVI:eqn2}, it suffices to establish that $v_{ij}^*=\Theta_{ij}[Q^*],~\forall (i,\,j)\in\mathcal{W}$. We proceed as follows. 

Fix arbitrary $(k,\,l)\in\mathcal{W}$ such that $Q^*_{kl}\neq 0$ \footnote{For any origin-destination pair $(i,\,j)$ such that $Q^*_{ij}=0$, the logical statements \eqref{chapVI:eqn1} and \eqref{chapVI:eqn2} automatically hold, since in this case $h_p^*(\cdot)$ vanishes almost everywhere for all $p\in\mathcal{P}_{ij}$.}. We define the pair $(\hat h,\,\hat Q)\in\widetilde\Lambda$ as 
\begin{align*}
\hat h_p(t)&~=~\begin{cases}a\,h_p^*(t)\qquad &p\in\mathcal{P}_{kl}\\ 
h_p^*(t) \qquad & p\in\mathcal{P}\setminus \mathcal{P}_{kl}\end{cases}\qquad\forall t\in[t_0,\,t_f] 
\\
\hat Q_{ij}&~=~\begin{cases} a\,Q^*_{kl}\qquad &(i,\,j)=(k,\,l)\\
Q^*_{ij}\qquad & (i,\,j)\in \mathcal{W}\setminus (k,\,l)\end{cases}
\end{align*}
where $a\in\Re_{++}$ is an arbitrary positive parameter.  Substituting $(h,\, Q)$ for $(\hat h,\,\hat Q)$, the left hand side of  \eqref{chapVI:eqn5} becomes
\begin{align*}
&\sum_{p\in \mathcal{P}}\int\nolimits_{t_{0}}^{t_{f}}\Psi _{p}(t,h^*)(\hat h_{p}-h_{p}^{\ast })dt
-\sum_{(i,\,j) \in \mathcal{W}}\Theta _{ij}\left[ Q^{\ast } \right] \left[\hat Q_{ij} -Q_{ij}^{\ast}
\right]
\\
~=~&\sum_{p\in\mathcal{P}_{kl}}\int_{t_0}^{t_f}\Psi_p(t,\,h^*)(ah_p^*-h_p^*)\,dt+\sum_{p\in\mathcal{P}\setminus\mathcal{P}_{kl}}\int_{t_0}^{t_f}\Psi_p(t,\,h^*)(h_p^*-h_p^*)\,dt
\\
&~-~\Theta_{kl}[Q^*]\left[aQ^*_{kl}-Q^*_{kl}\right]-\sum_{(i,\,j)\in\mathcal{W}\setminus (k,\,l)}\Theta_{ij}[Q^*]\left[Q_{ij}^*-Q^*_{ij}\right]
\\
~=~&(a-1)\sum_{p\in\mathcal{P}_{kl}}\int_{t_0}^{t_f}\Psi_p(t,\,h^*)\,h_p^*(t)\,dt-(a-1)\Theta_{kl}[Q^*]\,Q^*_{kl}
\\
~=~&(a-1) v_{kl}^* Q^*_{kl}-(a-1)\Theta_{kl}[Q^*]Q^*_{kl}
\end{align*}
We conclude from \eqref{chapVI:eqn5} that 
$$
(a-1)\left(v^*_{kl}-\Theta_{kl}[Q^*]\right)Q^*_{kl}~\geq~0
$$
Since $a\in\Re_{++}$ is arbitrary, there must hold $v_{kl}=\Theta_{kl}[Q^*]$, for any $(k,\,l)\in \mathcal{W}$. The proof is complete.\end{proof}

\section{The differential variational inequality formulation of E-DUE}\label{secdvi}
The notion of  dynamic user equilibrium can be alternatively illustrated and analyzed using the mathematical paradigm of differential variational inequality (DVI) \citep{Friesz2001, FKKR}. In particular, the demand satisfaction can be easily rewritten as a two-point boundary value problem
\begin{equation}\label{tbivp}
{d y_{ij}\over dt}~=~\sum_{p\in\mathcal{P}_{ij}}h_p(t),\qquad y_{ij}(t_0)~=~0,\qquad y_{ij}(t_f)~=~Q_{ij}\qquad\forall (i,\,j)\in\mathcal{W}
\end{equation}
where $y_{ij}(\cdot)$ is interpreted as the cumulative departure curve. In the context of the elastic demand case,  we assume there are \underline{unknown} terminal states  $
y_{ij}(t_{f})$, for all $\left( i,j\right) \in \mathcal{W}$, which are the
realized DUE travel demands. Moreover, for each origin-destination pair $%
\left( i,j\right) \in \mathcal{W}$, inverse travel demand is expressed as%
\begin{equation}
v_{ij}=\Theta _{ij}\left[ y\left( t_{f}\right) \right]  \label{tayta}
\end{equation}
where
$$
y~=~\big(y_{ij}(\cdot):~(i,\,j)\in\mathcal{W}\big)\in\Big(\mathcal{H}^1([t_0,\,t_f])\Big)^{|\mathcal{W}|}
$$
$\mathcal{H}^1([t_0,\,t_f])$ is the Sobolev space consisting of weakly differentiable functions whose weak derivatives are square-integrable. Thus, the $y_{ij}(t_{f})$, for all $\left( i,j\right) \in \mathcal{W}$, will
be determined endogenously to the differential variational inequality
presented subsequently in Theorem \ref{dvielasticthm}. Such an approach contrasts to
the approach employed by \cite{FKKR} to study fixed-demand DUE by
making each $y_{ij}(t_{f})$ an {\it a priori} fixed constant $Q_{ij}$. Accordingly, we
introduce the following dynamics:%
\begin{equation}
\frac{dy_{ij}}{dt}=\sum_{p\in \mathcal{P}_{ij}}h_{p}, \qquad
y_{ij}(t_{0})=0  \qquad  \forall \left( i,j\right) \in \mathcal{W}
\label{invap}
\end{equation}%
As a consequence, we employ the following alternative form of the feasible
set:%
\begin{equation}
\Lambda_1 =\left\{ h\geq 0:\frac{dy_{ij}}{dt}=\sum_{p\in P_{ij}}h_{p}\left(
t\right) \text{ \ \ \ }y_{ij}(t_{0})=0\text{ \ \ }\forall \left( i,j\right)
\in \mathcal{W}\right\} \subseteq \left( L_{+}^{2}\left[ t_{0},t_{f}\right]
\right) ^{\left\vert \mathcal{P}\right\vert }  \label{fregb}
\end{equation}%
Note that the feasible set $\Lambda_1$ in (\ref{fregb}) is expressed as a set
of path flows since knowledge of $h$ completely determines the demands that
satisfy the initial value problem (\ref{tbivp}). The differential variational inequality formulation of E-DUE is first stated and proved by \cite{FM2013} using optimal control theory.

\begin{theorem}\label{dvielasticthm}
{\bf (E-DUE equivalent to a differential
variational inequality)} Assume $\Psi _{p}(\cdot ,h):\left[ t_{0},t_{f}\right]
\longrightarrow \Re _{++}^{1}$ is measurable and strictly positive for all $%
p\in \mathcal{P}$ and all $h\in \Lambda_1 $. Also assume that the elastic
travel demand function is invertible, with inverse $\Theta _{ij}\left(
Q\right) $ for all $\left( i,j\right) \in \mathcal{W}$. A vector of
departure rates (path flows) $h^{\ast }\in \Lambda_1$ is a dynamic user
equilibrium with associated demand $Q^{\ast }(t_{f})$ if and only if $%
h^{\ast }$ solves the following $DVI(\Psi ,\Theta ,t_{0},t_{f})$.
\begin{equation}
\left.
\begin{array}{c}
\hbox{find}~h^*\in\Lambda_1~\hbox{such that}\\ 
\displaystyle \sum_{p\in \mathcal{P}}\int\nolimits_{t_{0}}^{t_{f}}\Psi _{p}(t,h^{\ast
})(h_{p}-h_{p}^{\ast })dt-\sum\limits_{\left( i,j\right) \in \mathcal{W}%
}\Theta _{ij}\left[ y^{\ast }\left( t_{f}\right) \right] \left[ \rule%
{0pt}{13pt}y_{ij}\left( t_{f}\right) -y_{ij}^{\ast }\left( t_{f}\right) %
\right] \geq 0\\
\forall h\in \Lambda_1
\end{array}
\right\} DVI\big(\Psi,\,\Theta,\,t_0,\,t_f\big)
 \label{duedvi}
\end{equation}
\end{theorem}
\begin{proof}
See \cite{FM2013}.
\end{proof}

\section{Existence of dynamic user equilibrium with elastic demand}\label{secexistenceedue}  

In this section, we will establish existence result for $VI(\Psi,\Theta,t_0,t_f)$, an equivalent formulation of E-DUE. Our proposed approach is meant to incorporate the most general dynamic network loading sub-model with minimum regularity requirements, and to yield existence of E-DUE without invoking the {\it a priori} upper bound on path flows. In subsequent analysis, we will rewrite $VI(\Psi,\Theta,t_0,t_f)$ as a variational inequality problem in an extended Hilbert space and then employ a version of the Brouwer's Fixed Point Theorem \citep{Browder}. To this end, we introduce the product space $ E\doteq \big(L^2[t_0,\,t_f]\big)^{|\mathcal{P}|}\times \Re^{|\mathcal{W}|}$, which is a Hilbert space with the induced inner product defined as follows 
\begin{align}
\label{inducedipdef}
\left<X,\,Y\right>_E~\doteq~\sum_{i=1}^{|\mathcal{P}|}\int_{t_0}^{t_f}\xi_i(t)\cdot \eta_i(t)\,dt+\sum_{j=1}^{|\mathcal{W}|}u_jv_j\\
X~=~\left(\xi_1(\cdot),\,\ldots,\,\xi_{|\mathcal{P}|}(\cdot),\, u_1,\,\ldots,\,u_{|\mathcal{W}|}\right)\in E\\
Y~=~\left(\eta_1(\cdot),\,\ldots,\,\eta_{|\mathcal{P}|}(\cdot),\, v_1,\,\ldots,\,v_{|\mathcal{W}|}\right)\in E
\end{align}
Let us recall the VI formulation presented in Section \ref{secvielastic}, the set $\widetilde\Lambda$ of  admissible pair $(h,\,Q)$ can now be embedded in the extended space $E$:
$$
\widetilde\Lambda~=~\left\{(h,\,Q)\in(L^2_+[t_0,\,t_f])^{|\mathcal{P}|}\times \Re_+^{|\mathcal{W}|}:~ \sum_{p\in\mathcal{P}_{ij}}\int_{t_0}^{t_f}h_p(t)\,dt~=~Q_{ij}\quad \forall (i,\,j)\in\mathcal{W} \right\}~\subset~E
$$
In view of the inverse demand function $\Theta=(\Theta_{ij}:\,(i,\,j)\in\mathcal{W})$, we introduce  notation
$$
\Theta^-~\doteq~(-\Theta_{ij}:\,(i,\,j)\in\mathcal{W}):\quad\Re_+^{|\mathcal{W}|}\longrightarrow \Re_-^{|\mathcal{W}|}
$$
We next define the mapping 
\begin{equation}\label{calfdef}
\mathcal{F}: \widetilde \Lambda \longrightarrow E,\qquad (h,\,Q)~\mapsto~\left(\Psi(\cdot,\,h),\, \Theta^-(Q)\right)
\end{equation}
where $(h,\,Q)\in\widetilde \Lambda$, $\Psi(\cdot,\,h)\in\big(L_+^2[t_0,\,t_f]\big)^{|\mathcal{W}|},\, \Theta^-(Q)\in\Re_-^{|\mathcal{W}|}$. Such a mapping is clearly well-defined.  With the preceding discussion, the VI formulation of the DUE problem with elastic demand is readily rewritten as the following infinite-dimensional variational inequality in the extended Hilbert space.

\begin{equation}\label{viextend}
\left. 
\begin{array}{c}
\text{find }X^{\ast }\in \widetilde\Lambda \text{ such that} \\ \\
\left<\mathcal{F}(X^*),\, X-X^*\right>_E~\geq~0 \\ \\
\forall X\in \widetilde\Lambda %
\end{array}%
\right\} VI\Big(\mathcal{F}, \Theta, \widetilde\Lambda, t_0, t_f\Big) 
\end{equation}%
where $X=(h,\,Q)$ and $X^*=(h^*,\,Q^*)$. Problem (\ref{viextend}) is written in the generic form of variational inequality, which allows analysis regarding solution existence to be carried out in a framework provided by \cite{Browder}.

\subsection{Existence result for $VI\big(F,\,\Theta,\,\widetilde\Lambda,\,t_0,\,t_f\big)$}\label{secexistence}

Our qualitative analysis regarding solution existence for the variational inequality (\ref{viextend}) is based on the following extension of Brouwer's fixed point theorem to topological vector spaces. 

\begin{theorem}\label{browderthm} 
\citep{Browder} Let $K$ be a compact convex subset of the locally convex
topological vector space $V$, $T$ a continuous (single-valued) mapping of $K$
into $V^{\ast }$, where $V^*$ is the dual space of $V$. Then there exits $u_{0}$ in $K$ such that 
\begin{equation*}
\Big<T(u_{0}),\,u_{0}-u\Big>~\geq ~0
\end{equation*}%
for all $u\in K$.
\end{theorem}
\begin{proof}
See \cite{Browder}.
\end{proof}

In preparation for our existence proof, we recap several key results from functional analysis that facilitate our presentation. In particular, we note the following facts. The reader is referred to \cite{Royden} for more detailed discussion on these subjects. 

\begin{proposition}
The space of square-integrable real-valued functions on a compact interval $[t_0,\,t_f]$, denoted by $L^2[t_0,\,t_f]$, is a locally convex topological vector space. In addition, the $|\mathcal{P}|$-fold product of these spaces, denoted by $\big(L^2[t_0,\,t_f]\big)^{|\mathcal{P}|}$, is also a locally convex topological vector space. 
\end{proposition}

\begin{proposition}
The dual space of $\big(L^{2}[t_0,\,t_f]\big)^{|\mathcal{P}|}$  has a natural isomorphism with $\big(L^{2}[t_0,\,t_f]\big)^{|\mathcal{P}|}$. The dual space of the Euclidean space $\Re^{|\mathcal{W}|}$ consisting of columns of $|\mathcal{W}|$ real numbers is interpreted as the space consisting of rows of $|\mathcal{W}|$ real numbers. As a consequence, the dual space of $\big(L^2[t_0,\,t_f]\big)^{|\mathcal{P}|}\times\Re^{|\mathcal{W}|}$ is again $\big(L^2[t_0,\,t_f]\big)^{|\mathcal{P}|}\times\Re^{|\mathcal{W}|}$.
\end{proposition}

\begin{proposition}\label{propsition3}
In a metric space (therefore topological vector space), the notion of compactness is equivalent to the notion of sequential compactness, that is, every infinite sequence has a convergent subsequence. 
\end{proposition}

Theorem \ref{browderthm} is immediately applicable for showing that $VI\big(F,\,\Theta,\,\widetilde\Lambda,\,t_0,\,t_f\big)$ has a solution if (1) $\mathcal{F}$ is continuous; and (2) $\widetilde \Lambda\subset E$ is compact. Unfortunately, such compactness does not generally hold for the problem we study herein. To overcome such an obstacle, we proceed in a similar way as in \cite{existence} by considering finite-dimensional approximations of the underlying infinite-dimensional Hilbert space. Another major hurdle that stymied many researchers is the {\it a priori} upper bound of path flows. Such bound is important for a topological argument that we will rely on in the proof, but does not arise from any physical or behavioral perspective of traffic modeling. In fact, as observed by \cite{BH}, the equilibrium path flows could very well become unbounded or even produce dirac-delta, if no additional assumptions are made regarding exogenous parameters of the Nash-like game, such as travelers' disutility functions.

The following assumptions are key to our analysis of path flows. The first assumption, {\bf (A1)}, poses hypothesis on drivers' perceived arrival costs; the second assumption, {\bf (A2)}, is concerned with the model of link/path dynamics and can be easily satisfied by existing  models such as the Vickrey model \citep{Vickrey, GVM1, GVM2}, the LWR-Lax model \citep{FHNMY}, and the Lighthill-Whitham-Richards model \citep{CTM, LW, Richards}. 
\\

\noindent {\bf (A1).} The function $f(\cdot )$ appearing in \eqref{cost} is
continuous on $[t_0,\,t_f]$ and satisfies 
\begin{equation}\label{fassumption}
f(t_2)-f(t_1)~\geq~\Delta (t_2-t_1) \qquad \forall t_0~\leq~t_1~<~t_2~\leq~t_f
\end{equation}
for some $\Delta >-1$\\

\noindent {\bf (A2).} The {\it first-in-first-out} (FIFO) rule is obeyed on a path level. In addition, each link $a\in \mathcal{A}$ in the network has
a finite exit flow capacity $M_{a}~<~\infty $. \\

\begin{remark}
Assumption {\bf (A1)} is employed in an in-depth analysis of the network flow later in the proof of Theorem \ref{existencethm}. The reader is referred to \cite{existence} for the motivation and generality of such an assumption. Notice that  if $f(\cdot)$ is continuously differentiable, then {\bf (A1)} is equivalent to requiring that $f'(t) > -1,~t\in[t_0,\,t_f]$. 
\end{remark}

\noindent With the preceding preparation, we are now ready to state and prove the main result of this section.

\begin{theorem}\label{existencethm}
Assume that the effective delay operator $\Psi: \Lambda_0\rightarrow \big(L^2[t_0,\,t_f]\big)^{|\mathcal{P}|}$ is continuous. In addition, let assumptions {\bf (A1)} and {\bf (A2)} hold. If the inverse demand function $\Theta: \Re_+^{|\mathcal{W}|}\rightarrow \Re_{++}^{|\mathcal{W}|}$ is continuous, then the variational inequality $VI\big(F,\,\Theta,\,\widetilde\Lambda,\,t_0,\,t_f\big)$ has a solution.
\end{theorem}

\begin{proof}
Given that $\Psi$ and $\Theta$ are both continuous mappings, it is straightforward to verify by definition \eqref{calfdef} that the  mapping $\mathcal{F}: \widetilde \Lambda \rightarrow E$ is also continuous. 

Since Theorem \ref{browderthm}  cannot be directly applied to obtain a solution of the infinite-dimensional VI, let us instead employ the technique from \cite{existence} by considering finite-dimensional approximations of $\widetilde\Lambda$. More specifically, consider for each $n\geq 1$ the uniform partition of $[t_0,\,t_f]$ by $n$ sub-intervals $I^1,\ldots, I^n$. Define the finite-dimensional subset of $\widetilde\Lambda$:
\begin{multline}
\widetilde\Lambda_n~\doteq~\Bigg\{\big(h_1(\cdot),\ldots, h_{|\mathcal{P}|}(\cdot),\,Q_1,\ldots, Q_{|\mathcal{W}|}\big)\in \widetilde\Lambda: 
\\
 h_i(\cdot)~\hbox{is  constant on each } I^j\qquad \forall 1\leq j\leq n,\quad \forall 1\leq i\leq |\mathcal{P}|   \Bigg\}
\end{multline}
\noindent Moreover, it is not restrictive to assume that there is an upper bound on the elastic demand for each origin-destination pair. That is, there exists a vector $U=\big(U_{ij}\big)\in \Re_{++}^{|\mathcal{W}|}$ such that 
$$
0~\leq~Q_{ij}~\leq~U_{ij} \qquad \forall (i,\,j)\in\mathcal{W},\qquad  \forall (h,\,Q)\in\widetilde\Lambda
$$

\noindent We can then show that $\widetilde\Lambda_n$ defined as such is convex and compact in $\widetilde\Lambda$ for each $n\geq 1$. The proof is postponed until Lemma \ref{appdxlemma1} of the Appendix for the clarity and concision of our presentation.

We are now in a position where Theorem \ref{browderthm} applies to each $\widetilde\Lambda_n$. In other words, fix $n\geq 1$, there exists some $X^{n,*}=(h^{n,*},\,Q^{n,*})\in\widetilde\Lambda_n$ such that 
\begin{equation}\label{chapExistence:eqn2}
\left<\mathcal{F}\left(X^{n, *}\right),\,X^n-X^{n,*}\right>_E~\geq~0\qquad\forall X^n\in\widetilde\Lambda_n
\end{equation}

\noindent As a consequence of  \eqref{chapExistence:eqn2}, we have
\begin{align}
&\sum_{p\in\mathcal{P}}\int_{t_0}^{t_f}\Psi_p\big(t,\,h^{n,*}\big) h^{n,*}(t)\,dt-\sum_{(i,\,j)\in\mathcal{W}}\Theta_{ij}[Q^{n,*}]Q_{ij}^{n,*} \nonumber
\\
\label{eqn1}
~\leq~&\sum_{p\in\mathcal{P}}\int_{t_0}^{t_f}\Psi_p\big(t,\,h^{n,*}\big) h^{n}(t)\,dt-\sum_{(i,\,j)\in\mathcal{W}}\Theta_{ij}[Q^{n,*}]Q_{ij}^n
\end{align}
for all $\big(h^n,\,Q^n\big)\in\widetilde\Lambda_n$. In particular, for all $h^n\in \big(L^2_+[t_0,\,t_f]\big)^{|\mathcal{P}|}$ such that 
$$
\sum_{p\in\mathcal{P}_{ij}}\int_{t_0}^{t_f}h_p^n(t)\,dt~=~Q^{n,*}_{ij}\qquad\forall (i,\,j)\in\mathcal{W}
$$
\noindent inequality \eqref{eqn1} becomes 
\begin{equation}\label{eqn2}
\sum_{p\in\mathcal{P}}\int_{t_0}^{t_f}\Psi_p\big(t,\,h^{n,*}\big) h^{n,*}(t)\,dt~\leq~\sum_{p\in\mathcal{P}}\int_{t_0}^{t_f}\Psi_p\big(t,\,h^{n,*}\big) h^{n}(t)\,dt
\end{equation}
Recall that $h^{n,*}(\cdot)$ is piecewise constant, thus \eqref{eqn2} implies the following: 
\begin{equation}\label{eqn3}
h^{n,*}_p(t)~>~0,\, \quad t\in I^k ~\Longrightarrow~ \int_{I^k}\Psi_p(t,\,h^{n,*}) \,dt ~=~\min_{1\leq j\leq n} \int_{I^j}\Psi_p(t,\,h^{n,*}) \,dt
\end{equation}
\noindent for all $p\in\mathcal{P}_{ij}$, $(i,\,j)\in\mathcal{W}$. We invoke Lemma \ref{appdxlemma2} from the Appendix and find a constant $\mathcal{M}\in \Re_{++}$ such that 
$$
h_p^{n,*}(t)~\leq~\mathcal{M}\qquad\forall t\in[t_0,\,t_f],\quad\forall p\in\mathcal{P},\quad\forall n\geq 1
$$
With the uniform upper bound $\mathcal{M}$ on the path flows $h^{n,*}$ and upper bound $U_{ij}$ for each $Q_{ij}$, $(i,\,j)\in\mathcal{W}$, one can extract a subsequence $\left\{X^{n_k,*}\right\}$ out of $\left\{X^{n,*}\right\}$ such that $X^{n_k,*}$ converges weakly to some $X^*\in\widetilde\Lambda$ as $k\rightarrow\infty$, where the weak topology on $E$ is defined via the inner product $\left<\cdot,\,\cdot\right>_E$. 

In view of \eqref{chapExistence:eqn2}, by strong continuity of $\mathcal{F}$ and the weak convergence $X^{n,*}\rightarrow X^*$, we conclude that
$$
\left<\mathcal{F}(X^*),\,X-X^*\right>_E~\geq~0\qquad\forall X\in\widetilde\Lambda
$$  
\end{proof}

\begin{remark}
The existence of an E-DUE requires that the embedded network loading procedure yields a delay operator that is continuous. Such regularity condition coincides with that of the fixed demand case, see \cite{existence}. Theorem \ref{existencethm} subsumes all notions of simultaneous-route-and-departure choice user equilibrium regardless of the arc dynamic, flow propagation and delay model employed. 
\end{remark}

\section{Conclusion}
We have shown that dynamic network user equilibrium based on simultaneous departure time and route choice in the presence of elastic travel demand may be formulated as a variational inequality (VI). Such a result is a nontrivial extension of the VI formulation established by \cite{Friesz1993} for dynamic user equilibrium with fixed travel demand. The proof relies on a measure-theoretic argument and provides unique insights on the qualitative properties of E-DUE such as existence, which is also established in this paper.

Similar to the fixed demand case, existence of the elastic demand case is easily analyzed using the framework of Brouwer's fixed point theorem for infinite-dimensional variational inequalities proposed in \cite{Browder}. Nevertheless, the variational inequality for the elastic demand case has both infinite- and finite-dimensional terms. The proof of existence requires construction of an extended space that subsumes both parts, as well as an appropriate choice of inner product  that allows compactness and weak topology to be defined.  It is significant that our proof does not rely on the {\it a priori} upper bound on path flows.

We are conducting further investigation of the E-DUE problem in terms of computation, based on the DVI formulation. In particular, the optimal control framework inherent in the DVI formulation allows us to device an iterative scheme based on a fixed-point-problem reformulation. Convergence of such a scheme requires some sort of monotonicity to be articulated and proved, such as those mentioned in \cite{Friesz1993}, which is the focus of ongoing research.

\section{Appendix}
\begin{lemma}\label{appdxlemma1} Assume that there exists a vector $U=\big(U_{ij}\big)\in \Re_{++}^{|\mathcal{W}|}$ such that 
$$
0~\leq~Q_{ij}~\leq~U_{ij} \qquad \forall (i,\,j)\in\mathcal{W},\qquad  \forall (h,\,Q)\in\widetilde\Lambda
$$
Then for each $n\geq 1$, the subset
\begin{multline}
\widetilde\Lambda_n~\doteq~\Bigg\{\big(h_1(\cdot),\ldots, h_{|\mathcal{P}|}(\cdot),\,Q_1,\ldots, Q_{|\mathcal{W}|}\big)\in \widetilde\Lambda: 
\\
 h_i(\cdot)~\hbox{is  constant on each } I^j\qquad \forall 1\leq j\leq n,\quad \forall 1\leq i\leq |\mathcal{P}|   \Bigg\}
\end{multline}
is compact and convex in $\widetilde\Lambda$. 
\end{lemma}
\begin{proof}
We begin with verifying convexity. Let $\widehat X=(\widehat h,\,\widehat Q)$ and $\overline{X}=(\overline h,\,\overline{Q})$ be any two elements of $\widetilde\Lambda_n$. Given any $\alpha\in(0,\,1)$, we have that
$$
\sum_{p\in\mathcal{P}_{ij}}\int_{t_0}^{t_f}\Big(\alpha\,\widehat h_p(t)\,dt+(1-\alpha) \, \overline h_p(t)\Big)\,dt~=~\alpha\, \widehat Q_{ij}+(1-\alpha)\, \overline Q_{ij}\qquad \forall (i,\,j)\in\mathcal{W}
$$
Moreover, $ \alpha\, \widetilde h_p(\cdot)+(1-\alpha)\,\overline h_p(\cdot)$ clearly remains constant on each sub-interval $I^j,\,j=1,\ldots, n$, for all $p\in\mathcal{P}$. We thus conclude that $\alpha \widehat X+(1-\alpha)\overline X\in \widetilde \Lambda_n$.

Next, let us investigate compactness. From now on let us fix $n\geq 1$. In light of Proposition  \ref{propsition3}, it suffices to establish sequential compactness for $\widetilde\Lambda_n$. We consider an arbitrary infinite sequence $\big\{X^k\big\}_{k\geq 1}\subset \widetilde\Lambda_n$ where $X^k=\big(h^k,\,Q^k\big)$. For each $k\geq 1$ and $p\in\mathcal{P}$, let $\mu_p^k=(\mu_{p,j}^k)\in \Re_+^{n}$ be such that
$$
\mu_{p,j}^k~=~h^k_p(t) \qquad t\in I^j,\qquad \forall j~=~1,\,\ldots,\,n
$$
We then define $\mu^k\in\Re_+^{n |\mathcal{P}|}$ to be the concatenation of all vectors $\mu_p^k,\,p\in\mathcal{P}$. We also notice that the vectors $\mu^k,\,k\geq 1$ are uniformly bounded by the constant 
$$
\max_{(i,\,j)\in\mathcal{W}} U_{ij} \cdot { n\over t_f-t_0 }
$$
Thus by the Bolzano-Weierstrass theorem, there exists a convergent subsequence $\big\{\mu^{k'}\big\}_{k'\geq 1}$. It is immediately verifiable that the corresponding subsequence $h^{k'}$ converges uniformly on $[t_0,\,t_f]$ and also in the $L^2$ norm. 

Moreover, by virtue of the uniform bounds $U_{ij},\,(i,\,j)\in\mathcal{W}$, there exists a further subsequence $\big\{k''\big\}\subset\big\{k'\big\}$ such that  $Q^{k''}$ is a convergent subsequence  according to the Bolzano-Weierstrass theorem. Thus, the subsequence $\big\{X^{k''}\big\}_{k''\geq 1}$ converges with respect to the norm induced by inner product \eqref{inducedipdef}. 
\end{proof}

\vspace{0.2 in}

\begin{lemma}\label{appdxlemma2}
Assume that {\bf (A1)} and {\bf (A2)} hold. Then there exists $\mathcal{M}>0$ such that for all $\big(h^{n,*},\,Q^{n,*}\big)$ satisfying \eqref{chapExistence:eqn2}, there hold
\begin{equation}\label{eqn4}
h_p^{n,*}(t)~\leq~\mathcal{M}\qquad  \forall t\in[t_0,\,t_f],\qquad \forall p\in\mathcal{P},\qquad\forall n\geq 1
\end{equation}
\end{lemma}
\begin{proof}
In view of {\bf (A2)}, we are prompted to define the following constant
$$
M^{max}~\doteq~\max_{a\in\mathcal{A}} M_a~<~+\infty
$$
where $\mathcal{A}$ is the set of links of the network. Recalling the constant $\Delta$ from {\bf (A1)}, we choose constant $\mathcal{M}$ such that
$$
\mathcal{M}~>~{3 M^{max}\over \Delta+1}
$$
\noindent We claim that \eqref{eqn4} holds for such $\mathcal{M}$.  Otherwise, if \eqref{eqn4} fails, there must exist some $m\geq 1,\, q\in\mathcal{P}$ and $1\leq j\leq m$ such that
$$
h_q^{m,*}(t)~\equiv~\lambda~>~\mathcal{M}\qquad t\in I^j
$$
Without losing generality, we assume that $j> 1$ and consider the interval $I^{j-1}$. By possibly modifying the value of the function $\Psi_q(\cdot,\,h^{m,*})$ at one point without changing the measure-theoretic nature of the problem, we obtain $t^*\in I^{j-1}$ such that
$$
\Psi_q\big(t^*,\,h^{m,*}\big)~=~\sup_{t\in I^{j-1}} \Psi_q\big(t,\,h^{m,*}\big)
$$
We denote by $\tau_q(t,\, h^{m,*})$ the time of arrival at destination of driver who departs at time $t$ along path $q$. According to the first-in-first-out (FIFO) principle, we deduce that $\forall t\in I^j$,
$$
(t-t^j)\lambda~\leq~\int_{t^*}^t h^{m,*}_q(t)\,dt~\leq~M^{max}\big(\tau_q(t,\,h^{m,*})-\tau_q(t^*,\,h^{m,*})\big)
$$
where $t^j$ is the left boundary of the interval $I^j$. We then have the following estimation:
\begin{align}
&\Psi_q(t,\,h^{m,*})-\Psi_q(t^*,\,h^{m,*})
\\
~=~&D_q(t,\,h^{m,*})+f\big(\tau_q(t,\,h^{m,*})- T_A\big)-D_q(t^*,\,h^{m,*})-f\big(\tau_q(t^*,\,h^{m,*})- T_A\big)   \nonumber
\\
~\geq~&\tau_q(t,\,h^{m,*})-\tau_q(t^*,\,h^{m,*})-(t-t^*)+\Delta\big(\tau_q(t,\,h^{m,*})-\tau_q(t^*,\,h^{m,*})\big)    \nonumber
\\
~=~&(\Delta+1)\big(\tau_q(t,\,h^{m,*})-\tau_q(t^*,\,h^{m,*})\big)-(t-t^*)    \nonumber
\\
\label{eqn5}
~\geq~&(\Delta+1){\lambda\over M^{max}}(t-t^j) -(t-t^*)\qquad \qquad \forall t\in I^j
\end{align}
Integrating \eqref{eqn5} with respect to $t$ over interval $I^j$ shows the following:
\begin{multline}\label{eqn6}
\int_{I^j}\Psi_q(t,\,h^{m,*})\,dt - (t^{j+1}-t^j)\Psi_q(t^*,\,h^{m,*})
\\
~\geq~{(t^{j+1}-t^j)^2\over 2}\cdot {(\Delta+1)\lambda \over M^{max}}+(t^{j+1}-t^j)\cdot\left(t^*-{t^j+t^{j+1}\over 2}\right)
\end{multline}
where $t^j,\,t^{j+1}$ are respectively the left and right boundary of $I^j$. Since $t^*\in I^{j-1}$, we have that 
$$
t^*-{t^j+t^{j+1}\over 2}~\geq~-{3\over 2}\big(t^{j+1}-t^j\big)
$$
With this observation, \eqref{eqn6} becomes
\begin{align*}
\int_{I^j}\Psi_q(t,\,h^{m,*})\,dt - (t^{j+1}-t^j)\Psi_q(t^*,\,h^{m,*})~\geq~&{(t^{j+1}-t^j)^2\over 2}\cdot {(\Delta+1)\lambda \over M^{max}}-{3\over 2}\big(t^{j+1}-t^j\big)^2
\\
~=~&{(t^{j+1}-t^j)^2\over 2}\left({(\Delta+1)\lambda\over M^{max}}-3\right)
\\
~>~&0
\end{align*}
This implies
$$
\int_{I^j}\Psi_q(t,\,h^{m,*})\,dt~>~\int_{I^{j-1}}\Psi(t,\,h^{m,*})\,dt
$$
which yields contradiction to \eqref{eqn3}. This finishes the proof.
\end{proof}

\end{document}